\documentclass[11pt]{amsart}

\usepackage[usenames,dvipsnames]{xcolor}

\usepackage{amsmath,amsthm,amsfonts,graphicx,amssymb,amscd}
\usepackage[all,cmtip]{xy}
\usepackage{graphicx, overpic, float}
\usepackage[mathscr]{euscript}
\usepackage{hyperref}
\usepackage{pxfonts}

\usepackage{enumerate}

\newtheorem{thm}{Theorem}[section]

\newtheorem{ques}[thm]{Question}

\newtheorem{cor}[thm]{Corollary}
\newtheorem{lem}[thm]{Lemma}

\theoremstyle{definition}

\usepackage{color}

\title{Semi-regular tilings of the hyperbolic plane}

\author{Basudeb Datta}

\author{Subhojoy Gupta}

\address{Department of Mathematics, Indian Institute of Science, Bangalore 560012, India.}

\email{dattab@iisc.ac.in, subhojoy@iisc.ac.in}

\begin{document}
\setcounter{tocdepth}{4}
%\newpage
\begin{abstract} 
A semi-regular tiling of the hyperbolic plane is a tessellation by regular geodesic polygons with the property that each vertex has the same \textit{vertex-type}, which is a cyclic tuple of integers that determine the number of sides of the polygons surrounding the vertex. We determine combinatorial criteria for the existence, and uniqueness, of a semi-regular tiling with a given vertex-type, and pose some open questions. 
\end{abstract}

\maketitle

\section{Introduction}

A \textit{tiling} of a surface is a partition into (topological) polygons (the \textit{tiles}) which are non-overlapping (interiors are disjoint) and such that  tiles  which touch, do so either at exactly one vertex, or along exactly one common edge. The \textit{vertex-type} of a vertex $v$ is a cyclic tuple of integers $[k_1,k_2,\ldots ,k_d]$ where $d$ is the degree (or valence) of $v$, and each $k_i$  (for $1\leq i\leq d$)  is the number of sides (the \textit{size}) of the $i$-th polygon around $v$, in either clockwise or  {counter-clockwise order}. A \textit{semi-regular} tiling on a surface of constant curvature (\textit{eg.}, the round sphere, the Euclidean plane or the hyperbolic plane) is one in which each polygon is regular, each edge is a geodesic, and  the vertex-type is identical for each vertex (see Figure 1).
Two tilings are  \textit{equivalent}  if they are combinatorially isomorphic, that is, there is a homeomorphism of the surface to itself that takes vertices, edges and tiles of one tiling to those of the other.  In fact,  two semi-regular tilings of the hyperbolic plane are equivalent if and only if there is an  isometry of the hyperbolic plane that realizes the combinatorial isomorphism between them (see Lemma \ref{cor-leml}).

Semi-regular tilings of the {Euclidean} plane are called \textit{Archimedean} tilings, and  have been studied from antiquity. It is known that there are exactly eleven  such tilings, up to scaling --  see \cite{Datta2}, and \cite{G-S} for an informative survey.  
%Note, that one the  so-called ``Archimedean tilings"  in which the tiles are squares, hexagons and $12$-gons, is \textit{not} a semi-regular tiling by our definition above since their cyclic order is not the same at every vertex.  However, there are distinct pair of tilings with vertex-type $[3,3,3,3,6]$, which are mirror-images of each other.  
Also classical is the fact that the semi-regular tilings of the round sphere  are the following : the boundaries of  five famed {Platonic solids} and thirteen {Archimedean solids}  (which are each \textit{uniform}, that is, the tiling has a vertex-transitive automorphism group), the often overlooked pseudo-rhombicuboctahedron (see \cite{Gerr}), and  two infinite families -- the prisms and antiprisms. 

For a semi-regular tiling of the hyperbolic plane, it  is easy to verify that the vertex-type  $\mathsf{k} = [k_1,k_2,\ldots ,k_d]$ of any vertex must satisfy
\begin{equation}\label{asum}
\alpha(\mathsf{k}) := \displaystyle\sum\limits_{i=1}^d \frac{k_i-2}{k_i} > 2
\end{equation}
since the sum of the interior angles of a regular hyperbolic polygon is strictly \textit{less} than those of its Euclidean counterpart. %around a vertex must sum up to $2\pi$. 
We shall call  $\alpha(\mathsf{k})$ the \textit{angle-sum} of the cyclic tuple $\mathsf{k}$. 

There are plenty of examples of semi-regular tilings of the hyperbolic plane. Indeed, the Fuchsian triangle groups $G(p,q)$ generated by reflections on the sides of a hyperbolic triangle with angles $\frac{\pi}{2}, \frac{\pi}{p}$ and $\frac{\pi}{q}$ generate a semi-regular tiling with vertex-type $[p^q] = [\underbrace{p,p,\ldots,p}_{\text{q times}}]$ whenever $\frac{1}{p} + \frac{1}{q} < \frac{1}{2}$ (see \cite{Kulkarnietal}).  This tiling is also called \textit{regular} since all tiles are the same regular $p$-gon; moreover, the automorphism group is vertex-transitive, so it is a uniform tiling.

More generally, as a consequence of a result in \cite{Edmondsetal},  given any cyclic tuple of \textit{even} integers $\mathsf{k} = [2m_1,2m_2,\ldots ,2m_d]$ satisfying the angle-sum condition \eqref{asum}, there exists a semi-regular tiling of the hyperbolic plane with vertex-type $\mathsf{k}$. This tiling is obtained by first generating a tiling by  reflections on sides of a hyperbolic $d$-gon with interior angles $\pi/m_1, \pi/m_2,\ldots,\pi/m_d$, and then taking the \textit{dual} tiling  (see Corollary \ref{eventype} in \S6).  Note that in \cite{Edmondsetal} they considered the dual problem, that is, the existence of a tiling where  each tile is a  $d$-gon with the vertices having prescribed valencies; the Fuchsian reflection groups described above are known as Dyck groups. \vspace{.1in}

\begin{figure}
 % % Requires \usepackage{graphicx}
  \centering
  \includegraphics[scale=.42]{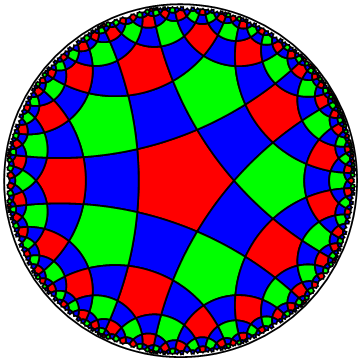}\vspace{.1in}
  \caption{A semi-regular tiling of the hyperbolic plane with vertex-type $\mathsf{k}= [4,5,4,5]$, a tuple that satisfies Condition (A). Theorem \ref{thm-uniq} asserts that in fact, this is the unique semi-regular tiling with this vertex-type.  }
\end{figure}

A basic question then is: 

\begin{ques}[Semi-regular Tiling Problem]\label{ques1}  Given a cyclic tuple of integers $\mathsf{k} = [k_1,k_2,\ldots ,k_d]$ satisfying $\alpha(\mathsf{k}) > 2$, does there exist a semi-regular tiling of the hyperbolic plane with vertex-type $\mathsf{k}$? If so, is the tiling unique?\end{ques}

\medskip

In this article we approach this as a \textit{combinatorial} problem and provide  \textit{sufficient} criteria for the existence and uniqueness of such semi-regular tilings for $d\geq 4$, and a complete list of vertex-types of semi-regular tilings with degree $d=3$.  Although the cases of degree $3$ and $4$ have been discussed earlier (see, for example, \cite{G-S-book} or \cite{Mitchell}), it is difficult to find complete proofs in the literature.   Note that semi-regular tilings, when treated as  combinatorial objects (namely, the graph obtained as a $1$-skeleton),   are also known as \textit{semi-equivelar maps} in the literature (see, for example, \cite{Datta2}). \vspace{.1in}

%To motivate our criteria, observe that for semi-regular tiling with vertex-type $\mathsf{k}$, if we also require that the polygons corresponding to the integers in $\mathsf{k}$ appear in \textit{counter-clockwise} order around each vertex, then there is an immediate necessary {combinatorial} condition that $\mathsf{k}$ needs to satisfy.  Namely, 

In what follows we say $u_1\cdots u_l$ \textit{appears} in a cyclic tuple $\mathsf{k}= [k_1, k_2,\ldots, k_d]$ if the integers $u, \cdots,u_l$ appear in consecutive order, that is,  $\mathsf{k}$ can be expressed in the form $[u_1,\ldots u_l, k_{l+1}, k_{l+2},\ldots k_d]$ or equivalently  $[u_l, u_{l-1}, \ldots, u_1, k_{l+1},  \ldots, k_{d}]$.
%the integers $u, \cdots,w$ appear in consecutive (either clockwise or anti-clockwise) order in $\mathsf{k}$.
We introduce the following simple combinatorial condition:

\begin{itemize}
\vspace{.1in}

%\item[(A)] \textit{ if $xy$ appears in $\mathsf{k}$ then either $xyx$ or $xy\cdots yx$ appears in $\mathsf{k}$.}

\item[(A)] \textit{ If $xy$ and $yz$ appear in the cyclic tuple  $[k_1, k_2,\ldots, k_d]$, then so does $xyz$.}
\vspace{.1in}
\end{itemize}

We then prove the following {sufficient} criterion for the existence of triangle-free semi-regular tilings:

\begin{thm}[Existence criteria - I]\label{thm1}
Consider a cyclic tuple  $\mathsf{k} = [k_1, k_2,... ,k_d]$ such that 
\begin{itemize}
\item the angle-sum $\alpha(\mathsf{k}) > 2$,
\item Condition (A) is satisfied, 
\item $d\geq 4$,  and  each $k_i\geq 4$.
\end{itemize}

Then there exists a semi-regular tiling of the hyperbolic plane with vertex-type $\mathsf{k}$.
%Moreover, this tiling is unique (up to the equivalence we defined) if any two  consecutive elements of the cyclic tuple uniquely determine the rest of the tuple. 
\end{thm}

%\noindent\textit{Remark.} It is easily verified that the angle sum condition $\alpha(\mathsf{k}) <2$ is automatic. For the excluded case, the angle sum \textit{equals} $2$, and it indeed corresponds to the familiar square tiling of the Euclidean plane. \vspace{.1in}

%Moreover, our construction will show:

%\begin{thm}[Uniqueness criterion]\label{uniq} There is a unique semi-regular tiling of the hyperbolic plane with vertex-type $\mathsf{k} =[k_1, k_2,... k_d]$, provided any two consecutive elements of the cyclic tuple uniquely determine the rest of the tuple. 
%\end{thm}

%As a corollary, we would then have:

%\begin{cor} The semi-regular tilings with vertex-type $[p^q]$ are unique, that is, any two such tilings differ by an isometry of the hyperbolic plane. 
%\end{cor}

For the case of tilings with triangular tiles, we consider degree $d\geq 6$, and an additional combinatorial condition:

\begin{itemize}
\vspace{.1in}

\item[(B)] \textit{If the triples $x3y$ and $3yz$ appear in the cyclic tuple $\mathsf{k}$, then so does $x3yz$.}
\vspace{.1in}

\end{itemize}

%Then, the technique of construction that we use for Theorem \ref{thm1} can be adapted to prove:

We shall prove:
\begin{thm}[Existence criteria - II]\label{thm2}
Consider a cyclic tuple  $\mathsf{k} = [k_1, k_2,... ,k_d]$ such that 
\begin{itemize}
\item the angle-sum $\alpha(\mathsf{k}) > 2$,
\item Conditions (A) and (B) are satisfied,  and 
\item $d\geq 6$. 
\end{itemize}

Then there exists a semi-regular tiling of the hyperbolic plane with vertex-type $\mathsf{k}$. 
%Moreover, this tiling is unique (up to the equivalence we defined)  if any two consecutive elements of the cyclic tuple uniquely determine the rest of the tuple.
\end{thm}

We also prove the following:

\begin{thm}[Uniqueness criteria]\label{thm-uniq}
 Let $\mathsf{k}$ be a cyclic tuple which satisfies the hypotheses of Theorem \ref{thm1} or Theorem \ref{thm2}. Suppose two consecutive elements  of $\mathsf{k}$  uniquely determine the rest of the cyclic tuple, i.e., given a pair $xy$ that appears in $\mathsf{k}$, there is a unique way of expressing $\mathsf{k}$ as the cyclic tuple $[x, y, k_3, \dots, k_d]$. Then there exists a unique semi-regular tiling $T$ whose vertex-type is $\mathsf{k}$. Moreover, in this case, the tiling  $T$ is uniform, i.e., it has vertex transitive automorphism group.
 %The semi-regular tiling $T$ obtained in Theorems \ref{thm1} and \ref{thm2} is unique (up to the equivalence we defined)  if any two consecutive elements of the cyclic tuple $\mathsf{k}$ uniquely determine the rest of the tuple. That is, given a pair $xy$ that appears in $\mathsf{k}$ in clockwise (respectively, counter-clockwise) order, there is a unique way of expressing $\mathsf{k}$ as the cyclic tuple $[x,y, k_3,k_4,\ldots, k_n]$, in clockwise (respectively, counter-clockwise) order.Moreover, in this case, the tiling $T$ is uniform, that is, has a vertex-transitive automorphism group. 
\end{thm}

As a corollary, we obtain the uniqueness of the regular tilings generated by the Fuchsian triangle-groups mentioned above (see \S4):

\begin{cor}\label{cor1} A semi-regular tiling of the hyperbolic plane with vertex-type $[p^q]$ (where $\frac{1}{p} + \frac{1}{q} < \frac{1}{2}$) is unique and uniform. In fact, any pair of such tilings are related by an isometry of the hyperbolic plane, that takes vertices and edges of one to vertices and edges, respectively, of the other. 
\end{cor}

\noindent\textit{Remark.} It is a folklore result that for a general vertex-type, uniqueness does not hold. Indeed, in \S5, we shall show that there could be \textit{infinitely many}  pairwise distinct semi-regular tilings of the hyperbolic plane with the same vertex-type. \vspace{.1in}

\smallskip

In \S6, we provide the following necessary and sufficient conditions for the existence of semi-regular tilings with degree $d=3$: 

\begin{thm}\label{thm0} 
A cyclic tuple $\mathsf{k} =[k_1,k_2,k_3]$ is the vertex-type of a semi-regular tiling of the hyperbolic plane if and only if one of the following holds:
\begin{itemize}
\item $\mathsf{k} = [p,p,p]$ where $p\geq 7$,
\item $ \mathsf{k} = [2n,2n,q]$ where $2n \neq q$, and $\frac{1}{n} + \frac{1}{q} < \frac{1}{2}$, or
    \item $ \mathsf{k} = [2\ell,2m,2n]$ where $\ell,m, n$ are distinct, and $\frac{1}{\ell} + \frac{1}{m} + \frac{1}{n} < 1$.
\end{itemize}
\end{thm}

\medskip

\begin{figure}
 % % Requires \usepackage{graphicx}
  \centering
  \includegraphics[scale=.4]{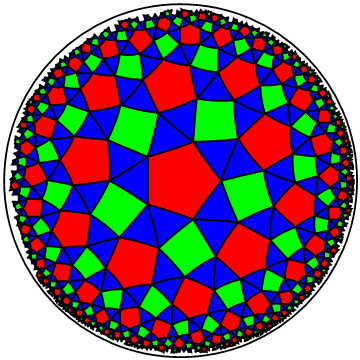}\vspace{.1in}
  \caption{A semi-regular tiling with vertex-type [5,3,4,3,3]. This cyclic tuple does not satisfy Condition (A). }
\end{figure}

We now mention some questions that are still open.\vspace{.1in}

First, our constructions for Theorems \ref{thm1} and \ref{thm2}  yields tilings that can have different symmetries, that is, could be invariant under different Fuchsian groups (or none at all) some of which have \textit{compact quotients}. 
%Indeed, most tilings in Theorem \ref{non-u} are not invariant under \textit{any} Fuchsian group.
Thus, we can ask:

\begin{ques} Given a cyclic tuple $\mathsf{k}=[k_1,k_2,\ldots, k_d]$, which \textit{compact} oriented hyperbolic surfaces have a semi-regular tiling with  vertex-type $\mathsf{k}$? How many such tilings does such a surface have? 
\end{ques}

%\noindent \noindent\textit{Remark.}
For  vertex-types of the form $[p^q]$ or $[2m_1,2m_2,\ldots ,2m_d]$, the question above was answered in \cite{Kulkarnietal} and \cite{Edmondsetal}, where it was shown that such a tiling exists whenever the appropriate Euler characteristic count holds. The case when the surface is a torus or Klein bottle, and  the vertex-type is that of a Euclidean tiling, was dealt with in \cite{Datta1b}; see also \cite{Wilson}, \cite{B-K} and \cite{P-W}. The work in \cite{Karabas} enumerates semi-regular tilings of surfaces of low genera which are also uniform. The recent work in \cite{Arun} also addresses the question above.   \vspace{.1in}

Second, the existence criterion, namely Condition (A),  in Theorem \ref{thm1} is not \textit{necessary} -- see, for example, Figure 2. For general degree $d$,  not \textit{all} tuples $\mathsf{k}$  that satisfy the angle-sum condition \eqref{asum} can be realized by a semi-regular tiling --  see \cite{Datta1c} for other necessary conditions. 

However, a set of necessary and sufficient conditions akin to Conditions (A) and (B) seems elusive. Although the work in \cite{Renault} develops algorithms for some related problems, we do not know if the answer to the following question is known:

\begin{ques} Is there a set of necessary and sufficient conditions that the cyclic tuple $\mathsf{k}$ needs to satisfy, to have a positive answer to the Semi-Regular Tiling Problem (see Question \ref{ques1})? Moreover, is the Semi-Regular Tiling Problem decidable? Namely, is there an algorithm to test if a given cyclic tuple $\mathsf{k}$ is the vertex-type of a semi-regular tiling of the hyperbolic plane, that terminates in finitely many steps?   \vspace{.1in}  
\end{ques}

\section{A tiling construction: Proof of Theorem \ref{thm1}}

In this section we prove Theorem \ref{thm1}, by describing a constructive procedure to tile the hyperbolic plane so that each vertex has the same vertex-type.
As we shall see, our method shall work provided the cyclic tuple $\mathsf{k}$ satisfies the hypotheses of Theorem \ref{thm1}, including Condition (A) mentioned in the introduction. Moreover, the algorithm will be free of choices under an additional hypothesis, proving the uniqueness statement of Theorem \ref{thm-uniq} in that case.

\subsection{Initial step: a fan} 
Let $\mathsf{k} = [k_1,k_2,\ldots, k_d]$ be a cyclic tuple of integers satisfying the hypotheses of Theorem \ref{thm1}.  Throughout, a \textit{closed hyperbolic disk} will mean a hyperbolic surface (i.e.\ with its interior supporting a smooth metric of constant curvature $-1$) homeomorphic to a closed disk. 
We shall begin with a closed hyperbolic disk $X_0$ constructed by attaching  $d$ regular polygons having numbers of sides $k_1,k_2, \ldots ,k_d$ respectively, to each other, around an initial vertex $V$.  In what follows, we shall call such a configuration of tiles around a vertex a \textit{fan}. \vspace{.1in}

A standard continuity argument implies:

\begin{lem}\label{leml}
There is a unique choice of a side-length $l_0>0$ for the polygons in $X_0$ such that the total angle around the vertex $V$ is exactly $2\pi$.
\end{lem}
\begin{proof}
For sufficiently small $l>0$, a regular hyperbolic polygon of $k_i$ sides and side length $l$ will be approximately Euclidean, and each interior angle will be close to $\pi {(k_i-2)}/{k_i}$. This makes the total angle $\theta(l)$  at vertex $V$ close  to $\pi \sum\limits_{i=1}^d \frac{k_i-2}{k_i}  >2\pi$, since the vertex-type satisfies the angle-sum condition to be a hyperbolic tiling. 
On the other hand, for large $l \gg 0$, as the vertices of the regular polygons tend to the ideal boundary,  each interior angle will be close to $0$, since for any ideal polygon adjacent sides bound cusps. The total angle  $\theta(l)$ is then close to $0$.
 In fact, elementary hyperbolic trigonometry  shows that $\theta$ is a continuous and  strictly monotonic function of $l$. Hence, there is a unique intermediate value $l_0 \in \mathbb{R}^+$ for which $\theta(l_0) = 2\pi$.
\end{proof}

\begin{figure}
 % % Requires \usepackage{graphicx}
  \centering
  \includegraphics[scale=.35]{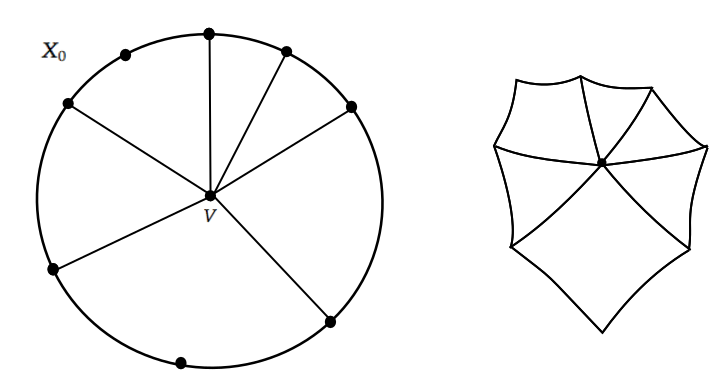}\vspace{.1in}
  \caption{A representation of a fan for the vertex-type [4,3,3,3,4,3]. The hyperbolic realization of this fan is shown on the right. }
\end{figure}

\noindent\textit{Remark.} Throughout this article, we shall use polygons with side-lengths equal to the $l_0$ obtained in the previous lemma, for a vertex-type $\mathsf{k}$. Moreover, we shall represent a fan as a closed disk with an appropriate division into wedges, together with vertices added to resulting boundary arcs to add more sides -- see Figure 3. From such a diagram it is straight-forward to recover the hyperbolic fan: we realize each of the resulting topological polygons as a regular hyperbolic polygon of side-length $l_0$. \vspace{.1in}

The first part of Question \ref{ques1} is equivalent to asking:
 \begin{ques} When can a hyperbolic fan be extended to a semi-regular tiling of the hyperbolic plane?
 \end{ques}

\medskip

It is easy to see that the following two properties hold for the tiled surface $X_0$. (Recall that by our assumptions on $\mathsf{k}$, there are no triangular tiles.) \vspace{.1in}

\noindent \textbf{Property 1.}\textit{ All boundary vertices have valence $2$ or $3$. Moreover, there is at least one boundary vertex of valence $2$ and one of valence $3$.} \vspace{.1in}

(Note that the second statement in fact follows from a weaker property that not \textit{all} tiles are triangles.)\vspace{.1in}

\noindent \textbf{Property 2.}  \textit{The tiled surface is a closed hyperbolic disk.}\vspace{.1in}

\subsection{Inductive step}

The tiling is constructed layer by layer, namely, we shall find a sequence of closed hyperbolic disks
\begin{equation*}
X_0 \subset X_1 \subset  X_2 \subset  \cdots  \subset X_i \subset X_{i+1} \subset \cdots 
\end{equation*}
each equipped with a tiling, such that their union $X_\infty$  is isometric to the entire hyperbolic plane, and the interior vertices of each $X_i$ have vertex-type $[k_1,k_2,\ldots , k_d]$.\vspace{.1in}

In the following construction, we shall describe how $X_{i+1}$ is obtained from $X_i$ by adding tiles around each boundary vertex of $X_i$ (that is, completing a fan), such that each boundary vertex of $X_i$ becomes an interior vertex of $X_{i+1}$.  Informally speaking, the tiles added to construct $X_{i+1}$ form a \textit{layer} around $X_i$; the final tiled surface $X_\infty$ is thus built by successively adding concentric layers. 

We shall now describe the inductive step of the construction,  namely, how to add tiles to expand from $X_i$ to $X_{i+1}$.\vspace{.1in}

To ensure that Property 2 is maintained, we shall repeatedly use the following elementary topological fact:

\begin{lem}\label{toplem} Let  $\Omega_0, \Omega_1$ be two topological spaces, each homeomorphic to a closed disk, 
and let $X = (\Omega_0 \cup \Omega_1)/{\sim}$ be the space obtained by identifying a connected non-trivial arc in $\partial \Omega_0$ with such an arc in $\partial \Omega_1$. Then $X$ is also homeomorphic to a closed disk. 
 \end{lem}

In the construction, we shall assume that Properties 1 and 2  hold for $X_i$. 
As we saw, these were true for $i=0$, namely for the tiled surface $X_0$. We shall verify it for $X_{i+1}$ when we complete the construction.\vspace{.1in}

As a consequence of Property 2, the boundary $\partial X_i$ is a topological circle. 
Let the boundary vertices be $v_0, v_1, \ldots , v_n$ in a counter-clockwise order.  Note that the number $(n+1)$ of vertices  is certainly dependent on $i$, and in fact grows exponentially with $i$, but we shall suppress this dependence for the ease of notation.\vspace{.1in}

Moreover, we shall choose this cyclic ordering such that $v_0$ is a vertex of valence $3$, and $v_n$ is a vertex of valence $2$.  (This is possible because Property 1 holds for $X_i$.) 
%The final property that we shall assume, which is obvious for $X_0$ and we shall check inductively is:

%\textit{Property 3.}  If a boundary vertex of $X_i$ is a vertex of two 

\subsection*{Completing a fan at $v_0$} 

We begin by adding tiles to complete the fan $F_0$ at  $v_0$.  

\begin{figure}[h]
 % % Requires \usepackage{graphicx}
  \centering
  \includegraphics[scale=.5]{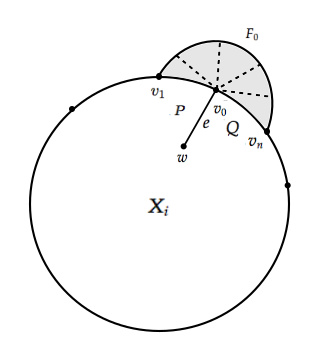}\vspace{.1in}
  \caption{The fan $F_0$ is completed at the boundary-vertex $v_0$ of $X_i$. The added wedge is shown shaded. }
\end{figure}

%There are two cases:\vspace{.1in}

%\textit{Case 1: The boundary vertex $v_0$ has valence 2.}  In this case, let $v_0$ be a vertex of a polygon $P$ in $X_i$.  The topological operation of adding the fan can be viewed as follows: consider a semi-circular arc centered at $v_0$ in the exterior of $X_i$ and with endpoints at $v_{1}$ and $v_{n}$. We add $d-2$ ``spokes" to the resulting semi-disk containing $v_0$: this results in a fan around $v_0$ comprising the initial polygon $P$, and exactly $d-2$ triangles.  Finally, we add more valence $2$ vertices to the sub-arcs of the semi-circle, if need be,  in order to obtain $d$ polygons with the desired number of sides and cyclic order prescribed by the vertex-type.\vspace{.1in}

%\textit{Case 2: The boundary vertex $v_0$ has valence 3.} 

Recall that $v_0$ has valence $3$ in $X_i$. Hence, $v_0$ is the common vertex of two polygons $P$ and $Q$ in $X_i$. 
The topological operation of adding the fan can be viewed as follows: consider a semi-circular arc centered at $v_0$ in the exterior of $X_i$ and with endpoints at $v_{1}$ and $v_{n}$. We add $d-3$ ``spokes" to the resulting ``wedge" containing $v_0$: this results in a fan around $v_0$ comprising the initial polygons $P$ and $Q$, and exactly $d-2$ triangles. 
 Finally, we add more valence $2$ vertices to the sub-arcs of the boundary of the wedge in order to obtain $d$ polygons with the desired sizes and cyclic order prescribed by the vertex-type. 
 
More explicitly: let $e$ be the common edge between $P$ and $Q$, one of whose endpoints is $v_0$, and the other, say, $w$. (See Figure 4.) Note that if the sizes of $P$ and $Q$ are  $x$ and $y$ respectively, then $xy$ appears (in clockwise order) in the  vertex-type for the vertex $w$. Then the vertex-type $\mathsf{k}$ can be expressed as  $[x,y, k_3,k_4,\ldots ,k_d]$, by recording the size of each polygon around $w$ when traversed in clockwise order.  Thus the cyclic tuple  $\mathsf{k}$ can also be expressed as $[y,x,k_d,k_{d-1}, \ldots, k_3]$.  Thus we can subdivide the wedge to obtain polygons  of sides  $k_d, k_{d-1}, \ldots ,k_3$ this time placed in counter-clockwise order around $w$, to complete the fan around $v_0$.

\subsection*{Completing fans at $v_1,v_2,\ldots ,v_{n-1}$}

We then successively complete fans  $F_j$ at $v_j$ for $1\leq j\leq n-1$ as follows. Assume we have completed fans at $v_0,v_1,\ldots ,v_{j-1}$.  At the $j$-th stage, there are two cases:\vspace{.1in}

\textit{Case I. The vertex $v_j$ has valence $3$ in $X_i \cup F_0\cup F_1 \cup \cdots \cup F_{j-1}$.}  This implies that $v_j$ had valence $2$ in $X_i$; the additional edge incident to $v_j$ comes from the fan $F_{j-1}$  added at $v_{j-1}$. Let $P$ be the polygon in $X_i$ that has $v_j$ as a vertex, and let $Q$ be the polygon in the fan  $F_{j-1}$ that has $v_j$ as a vertex.  Note that the edge between $v_{j-1}$ and $v_j$ is the common edge of $P$ and $Q$.  To describe the fan $F_j$ topologically, draw an arc in the exterior of $X_i \cup F_0 \cup F_1 \cup \cdots \cup F_{j-1}$, between $v_{j+1}$ and the vertex in $\partial F_{j-1}$ adjacent to $v_j$. (See Figure 5.)  Divide this wedge region into $d-2$ triangles by adding spokes, and as before, add an appropriate number of vertices to the resulting circular arcs to have polygons with more than three sides. Note that if the sizes of $P$ and $Q$ are $x$ and $y$ respectively, then $yx$ appears in the vertex-type of $v_{j-1}$, and as before, there is a choice of  such polygons that completes the fan $F_j$.\vspace{.1in}

\begin{figure}[h]
 % % Requires \usepackage{graphicx}
  \centering
  \includegraphics[scale=.4]{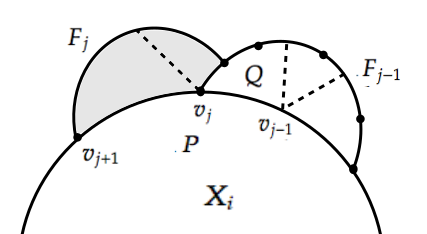}\vspace{.1in}
  \caption{The fan $F_j$ added at  $v_j$, that has valence $3$ after $F_{j-1}$ was added. }
\end{figure}

\textit{Case II. The vertex $v_j$ has valence $4$ in $X_i \cup F_0\cup F_1 \cup \cdots \cup F_{j-1}$.}  Then $v_j$ has valence $3$ in $X_i$. Let $P$ and $Q$ be the polygons in $X_i$ sharing the vertex $v_j$, such that $Q$ shares an edge with $F_{j-1}$, and let $P$ and $Q$ have $x$ and $y$ sides respectively.  (See Figure 6.) Then $xy$  appears in the vertex-type of $v_j$.  Let $R$ be the polygon in the fan $F_{j-1}$ that also has $v_j$ as a vertex, and let $z$ be the number of its sides. Note that by the assumption that the degree $d\geq 4$, the polygons $P,Q,R$ cannot be the only polygons around $v_j$ in the final tiling; our task is to show that we can add more to complete the fan at $v_j$. 

\begin{figure}[h]
 % % Requires \usepackage{graphicx}
  \centering
  \includegraphics[scale=.4]{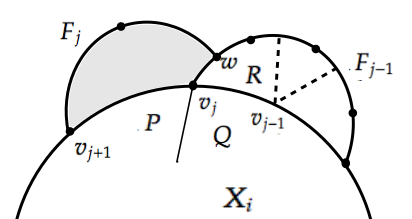}\vspace{.1in}
  \caption{The fan $F_j$ added at  $v_j$ of valence $4$: here we need Condition (A). }
\end{figure}

\noindent\textit{Claim 1. The triple $xyz$ appears in the cyclic tuple $\mathsf{k} = [k_1,k_2,\ldots, k_d]$.}
\vspace{.05in}

\indent The edge between $v_{j-1}$ and $v_j$ is common between $Q$ and $R$, and the vertex-type at $v_{j-1}$ includes $zy$. %By Condition (A), the tuple $\mathsf{k}$ will have $yz$ also. 
%We have already seen above that $\mathsf{k}$ contains $xy$ in the same cyclic order (\textit{viz.} counter-clockwise). 
We have already seen above that $xy$ appears in $\mathsf{k}$. 
Hence, by Condition (A),  the triple $xyz$ appears in $\mathsf{k}$. This proves Claim 1. \vspace{.1in}

 Hence, we can choose numbers of sides of successive polygons to follow $P,Q$ and $R$ around $v_j$,  to complete a fan $F_j$. We do this by adding and subdividing a wedge, just as in Case I.

\subsection*{Completing a fan at $v_{n}$} 

Finally, we need to complete the final fan $F_n$ around $v_n$. Note that we had chosen $v_n$ to have valence $2$ in $X_i$; after adding the fans $F_j$ for $0 \leq j \leq n-1$, $v_n$ has valence $4$ in $X_i \cup \left(\bigcup_{j=0}^{n-1} F_j\right)$, where the two additional edges belong to $F_0$ and $F_{n-1}$. Consider the three polygons $Q$, $P$ and $R$ in counter-clockwise order around $v_n$, where $Q$ is a polygon in $F_0$, $P$ is a polygon in $X_i$, and $R$ is a polygon in $F_{n-1}$, each having $v_n$ 
as a vertex.  Let $P,Q,R$ have $x,y,z$ sides respectively.  (See Figure 7.) \vspace{.1in}

\noindent\textit{Claim 2. The triple $yxz$ appears in the cyclic tuple $\mathsf{k} = [k_1,k_2,\ldots, k_d]$.}
\vspace{.05in}

\indent Since $P$ and $Q$ share the edge between $v_n$ and $v_0$, and since they are successive polygons (in counter-clockwise order) in the fan of $v_0$, $xy$ appears in the vertex-type $\mathsf{k}$. Similarly, $R$ and $P$ share the edge between $v_{n-1}$ and $v_n$, and are successive polygons in the fan of $v_{n-1}$, and so $zx$ appears in the vertex-type. Since both $yx$ and $xz$ appear in the vertex-type $\mathsf{k}$, by Condition (A), so does $yxz$. This proves Claim 2.\vspace{.1in}

Hence by adding a wedge based at $v_n$ between the fans $F_0$ and $F_{n-1}$, and subdividing into polygons, there is a choice of numbers of sides such that the successive polygons $Q,P,R$ are completed to a fan $F_{n}$.  Note that once again, we have implicitly used the assumption that $d\geq 4$, as in the process we are ending up with at least four polygons around $v_n$.

\begin{figure}[h]
 % % Requires \usepackage{graphicx}
  \centering
  \includegraphics[scale=.4]{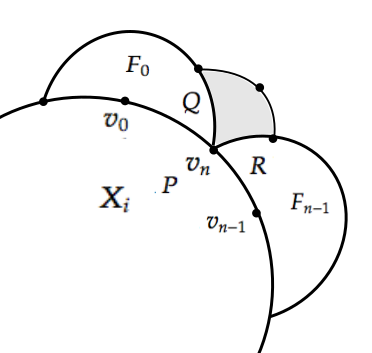}\vspace{.1in}
  \caption{The final fan $F_n$ involves adding a wedge between $F_0$ and $F_{n-1}$. }
\end{figure}

\subsection*{Verifying Properties 1 and 2}

We can now define 
\begin{equation*}
X_{i+1} := X_i \cup F_0 \cup F_1 \cup \cdots \cup F_n.
\end{equation*}
By construction, all the boundary vertices $v_0,v_1,\ldots ,v_n$ are in the interior of $X_{i+1}$, and every interior vertex has vertex-type $\mathsf{k}$. 

By the inductive hypothesis, $X_i$ is topologically a closed disk, and by construction, at each step of adding a fan at a boundary vertex, one is adjoining a simply-connected wedge to a connected arc of the boundary. Hence by Lemma \ref{toplem}, the union $X_i \cup F_0 \cup F_1 \cup \cdots \cup F_j$ is topologically a closed disk for each $j=0,1,2,\ldots ,n$. By the remark following Lemma \ref{leml}, we can realize each (topological) polygon as regular hyperbolic polygon of side length $l_0$, such that at each vertex the total angle is $2\pi$, and we have a smooth hyperbolic metric in the interior of $X_{i+1}$. This establishes Property 2 for $X_{i+1}$.   \vspace{.1in}

To check Property 1 for $X_{i+1}$, notice that the new boundary vertices are the vertices of the fans $F_0, F_1,\ldots ,F_n$ that lie on the boundary arcs of the wedges that we added.  Vertices that lie in the \textit{interior} of such a boundary arc have valence $2$ or $3$, exactly as for boundary vertices of a fan $X_0$.  There must be one such vertex of valence $2$, since there  is no triangular tile. 

Now a vertex $w$ that lies at the intersection of two adjacent wedges, say $F_j$ and $F_{j-1}$,  is the endpoint of the edge from $v_{j}$ to $w$ that is common to $F_j$ and $F_{j-1}$.  (See Figure 6.) Note that there cannot be an edge from $w$ to $v_{j+1}$, since then $v_{j+1}wv_j$ will form a triangular tile, contradicting our assumption that our tiling is triangle-free. Similarly, there cannot be an edge from $w$ to $v_{j-1}$. Hence this boundary vertex $w$ is of valence $3$.

Thus, all boundary vertices of $X_{i+1}$ have valence $2$ or $3$, and there is at least one vertex of valence $2$ and $3$, verifying Property 1.

\subsection{The endgame}
It only remains to show:

\begin{lem}\label{exh} The union $X_\infty$ of the tiled surfaces $X_i$ for $i\geq 0$ is isometric to the entire hyperbolic plane.
\end{lem}
\begin{proof}
Clearly, the initial fan $X_0$ contains a hyperbolic  disk of some positive radius, say  $r_0>0$, that is centered at the central vertex $V$.\vspace{.1in}

\noindent\textit{Claim 3. There is an $r>0$ such that any point on the boundary of  $X_{i+1}$ is at distance at least $r$ from $X_i$, for each $i\geq 0$.}

\vspace{.05in}

Any regular polygon in the annular region $X_{i+1}  \setminus X_i$  has either a  vertex or a side lying on the boundary of $X_i$. 
For each such polygon $P$, we choose $\delta_P >0$ as follows, in the two possible cases:
\begin{itemize}
\item[(i)]  $P \cap X_i = \{v\}$ where $v$ is a vertex of $P$ that is also in $\partial X_i$. Let $s_+$ and $s_-$ be the edges of $P$ that have $v$ as a common vertex. Let $N_{\delta}(v)$ be the set of points of $P$ at distance less than $\delta$ from $v$. We choose (sufficiently small) $\delta_P$ such that  $N_{\delta_P}(v)$ is at distance at least $\delta_P$ from the remaining sides of $P$ (that is, except $s_\pm$). 
\item[(ii)]   $P \cap X_i = \{s\}$, where $s$ be a side of $P$ lying in $\partial X_i$. Let $s_+$ and $s_-$ be the two sides adjacent to $s$ in $\partial P$.  Let $N_{\delta}(s)$ be the set of points of $P$ at distance less than $\delta$ from $s$.  We then choose (sufficiently small) $\delta_P$ such that $N_{\delta_P}(s)$ is at distance at least $\delta_P$ from the sides of $P$ other than $s,s_-,s_+$.
\end{itemize}
Since $P$ is a regular polygon (of side length $l_0$ -- see the remark following Lemma \ref{leml}),  the quantity $\delta_P$ only depends on the number of sides (the size) of $P$. However, the size of any such regular polygon is one of the integers in $\mathsf{k}$, and hence, by taking a minimum, we get a number $r>0$ that works for all the polygons $P$.

Now, for each polygon $P$ in  $X_{i+1}  \setminus X_i$, we consider the set $N_r(v)$ if $P \cap \partial X_i = \{v\}$ or $N_r(s)$ if $P\cap \partial X_i = \{s\}$. The union of these neighborhoods forms an annular region comprising points at distance less than $r$ from $\partial X_i$, that is disjoint from $\partial X_{i+1}$.  This proves Claim 3.\vspace{.1in}

Thus,  for each $i\geq 0$, the distance of the boundary of $X_{i}$ from $V$ is at least $r_0 + i\cdot r$, and hence the region $X_i$ includes a hyperbolic disk of radius $r_0 + i\cdot r$ centered at $V$. As $i\to \infty$, the radius tends to infinity, and hence the simply-connected surface $X_\infty$  we obtain in the union of the tiled surfaces is complete. By construction, this simply-connected surface has a smooth metric of constant negative curvature, and hence by the Cartan-Hadamard theorem, it is the entire hyperbolic plane. This proves the lemma. 
\end{proof}

This completes the proof of Theorem \ref{thm1}.\vspace{.1in}

 For the uniqueness statement, we need additional conditions on $\mathsf{k}$ to ensure that there is no choice in any step of the preceding construction: see Theorem \ref{thm-uniq} and its proof in \S4. \vspace{.05in}
 
 We conclude this section with the following observation:

\begin{lem}\label{cor-leml}
Let $\mathcal{T}$ be a (topological) tiling of the hyperbolic plane $\mathbb{H}^2$ with the property that the vertex-type at each vertex is $\mathsf{k}$, where $\mathsf{k}$ satisfies the angle-sum condition \eqref{asum}. Then $\mathcal{T}$ can be ``geometrized", that is, there is a semi-regular tiling of $\mathbb{H}^2$ by regular polygons with vertex-type $\mathsf{k}$ at each vertex, that is equivalent to $\mathcal{T}$. Moreover, any two semi-regular tilings equivalent to $\mathcal{T}$ are related by a hyperbolic isometry that takes vertices and edges of one to those of the other. 
\end{lem}
\begin{proof}
Realize each polygon of $\mathcal{T}$ as a regular polygon of side-length $l_0$, obtained in Lemma \ref{leml}, and attach them by isometric identifications of the sides  in the pattern dictated by $\mathcal{T}$.  By Lemma \ref{leml}, these polygons fit together at each vertex to form a fan, and so the resulting surface acquires a (smooth) hyperbolic metric. Moreover, by the argument in Lemma \ref{exh}, this surface is in fact complete, and since it tiles a simply-connected region, it is isometric to $\mathbb{H}^2$ by the Cartan-Hadamard theorem.  This results in a semi-regular tiling of $\mathbb{H}^2$ that is equivalent to $\mathcal{T}$. 

Let  ${T}$ and ${T}^\prime$ be two such semi-regular tilings equivalent to $\mathcal{T}$. Then since $T$ and $T^\prime$ are equivalent, there is a homeomorphism $h$  of the hyperbolic plane to itself,  that maps vertices, edges  and tiles of $T$ to those of $T^\prime$. Again by Lemma \ref{leml}, the choice of  hyperbolic length $l_0$ of the edges for a semi-regular tiling with vertex-type $\mathsf{k}$ is unique.  Thus, $h$ must  be length-preserving on each edge, and this can be extended to be an isometry on each tile.  Thus, we in fact have an isometry of the hyperbolic plane to itself, that realizes the equivalence between $T$ and $T^\prime$.
\end{proof}

 \section{Handling triangular tiles: Proof of Theorem \ref{thm2}}
 
Suppose we now have a cyclic tuple $\mathsf{k} = [k_1,k_2,\ldots , k_d]$ that satisfies the hypotheses of Theorem \ref{thm2}. This time,  $d\geq 6$, but we could have $k_i =3$ for some (or all) $i \in \{1,2,\ldots, d\}$.

\subsection*{Constructing the exhaustion} The construction  is the same inductive  procedure as in \S2: we start with a fan $X_0$ around a single vertex $V$, and proceed to build a sequence of tiled surfaces
 \begin{equation*}
 X_0 \subset X_1 \subset X_2 \subset \cdots \subset X_i \subset X_{i+1}  \subset \cdots 
 \end{equation*}
 which builds up a tiled surface $X_\infty$ that is isometric to the hyperbolic plane, and such that each interior vertex of $X_i$ has vertex-type $\mathsf{k}$, for each $i\geq 0$.
  
 In what follows we shall point out some of the differences with the proof of Theorem \ref{thm1} in \S2.  \vspace{.1in}
  
  The key difference is that this time a vertex of valence $4$ may appear on the boundary of a tiled surface  $X_{i+1}$ after completing the fans for the boundary vertices of $X_i$. (See Figure 8.) \vspace{.1in}

  Each region $X_i$ shall satisfy Property 2 as before, but the following different analogue of Property 1:\vspace{.1in}
  
 \noindent \textbf{Property $1^\prime$.}\textit{ 
  The following properties hold for $X_i$:}
  \begin{itemize}
  \item[(i)] \textit{ All boundary vertices in $\partial X_i$ have valence $2,3$ or $4$ in $X_i$.} \textit{ Moreover, there is at least one vertex of valence  at least $3$.}
   \item[(ii)]  \textit{Any  boundary vertex  $v \in \partial X_i$ of valence $4$ is the vertex of a triangular tile in $X_i$ that intersects the boundary $\partial X_i$ only at $v$.}
  \item[(iii)] \textit{Either there is a boundary vertex of valence $2$, or there is a boundary edge that belongs to a triangular tile.}
 
  \end{itemize}

  It is easy to see that $X_0$ satisfies (i) and (iii) above. Also,  (ii) is vacuously true as $\partial X_0$ does not have any vertex of valence $4$.  As mentioned,  valence $4$ vertices  may arise on the boundary of $X_i$ for $i\geq 1$ because of the presence of triangular tiles.\vspace{.1in}

%   \subsubsection*{Constructing $X_{i+1}$ from $X_i$} 
As before,  the boundary $\partial X_i$ is a topological circle because of Property 2, and we denote the boundary vertices of $\partial X_i$  by $v_0, v_1,\ldots ,v_n$ in counter-clockwise order.  We also require that:

\begin{itemize}
\item[(a)]  $v_0$ has valence $3$ or valence $4$. 
\item[(b)]  One of the two hold:
\begin{itemize}
\item Either $v_0v_n$ is a boundary edge that belongs to a triangular tile, and $v_n$ has valence $3$, or
\item $v_n$ has valence $2$ in $X_i$.
\end{itemize}
\end{itemize}

This is possible for $X_0$ since if $\mathsf{k} = [3^d]$, then each boundary vertex has valence $3$, and the first condition of (b) holds. Otherwise, as in \S2, we can in fact  choose $v_0$ to have valence $3$ and $v_n$ to have valence $2$, that is, satisfying the second condition of (b). For $X_i$, where $i\geq 1$, we shall verify that such a choice of $v_0$ and $v_n$ is possible at the end of the inductive step.\vspace{.1in}

%(We shall see that with our assumption of $d\geq 6$ we shall not need the restrictions on the valencies of $v_0$ and $v_n$, that we had for the construction for Theorem \ref{thm1}.)\vspace{.1in}

In what follows we shall complete fans  $F_j$  around $v_j$ for each $0\leq j\leq n$ as before, and define
\begin{equation*}
X_{i+1} := X_i \cup F_0 \cup F_1 \cup \cdots \cup F_n
\end{equation*}
for each $i\geq 0$.

\subsection*{Completing the fan at $v_0$}

When the valence of $v_0$ is $3$, we complete the fan $F_0$ around $v_0$ exactly as in the construction for Theorem \ref{thm1}.  When the valence of $v_0$  equals $4$, then the three polygons of $X_i$ around $v_0$ have sizes $x,3$ and $y$ in counter-clockwise order because of part (ii) of Property $1^\prime$. 

We need to ensure that we can continue placing polygons around $v_0$ to complete a fan, that is, we need to prove:\vspace{.1in}

\noindent\textit{Claim 4. The triple $x3y$ appears in the cyclic tuple $\mathsf{k} = [k_1,k_2,\ldots , k_d]$.}

\vspace{.05in}
Let $v_0w$ and $v_0w^\prime$ be the two edges at $v_0$ whose other endpoints are in the interior of $X_i$. Then Property $1^\prime$ (ii) implies that $v_0ww^\prime$ is a triangular tile. Suppose $P$ and $Q$ are the other polygons  in $X_i$ with $v_0$ as a vertex, having $x$ and $y$ sides respectively. (See Figure 8 for a similar situation where $v_j$ is the vertex, instead of $v_0$.) Then considering the vertices $w$ and $w^\prime$ that lie in the interior of $X_i$ (and consequently have vertex-type $\mathsf{k}$) we see that  the pairs $3x$ and $y3$ must appear in the cyclic tuple $\mathsf{k}$ (in counter-clockwise order). Hence by Condition (A), the triple $x3y$ appears in $\mathsf{k}$. This proves Claim 4.\vspace{.1in}

Then, as before, we can add a wedge at $v_0$ in the exterior of $X_i$, and subdivide into polygons by adding spokes and vertices on the resulting boundary arcs of the wedge, having the numbers of sides that determine the rest of the tuple $\mathsf{k}$ following $x, 3$ and $y$.  This completes the fan $F_0$ at $v_0$.

\subsection*{Completing the fan at $v_j$}
Now suppose we have completed fans around $v_0, v_1,\ldots, v_{j-1}$ for $1\leq j\leq n-1$, and we need to  complete the fan $F_j$ at $v_j$.

As before, our analysis divides into cases depending on the valency of the vertex $v_j$ in $X_i$.  When $v_j$ has valence $3$ or $4$ in the already-tiled surface $X_i \cup F_0 \cup F_1\cup \cdots \cup F_{j-1}$ , the completion of the fan $F_j$ proceeds exactly as in the corresponding step in the proof of Theorem \ref{thm1}. 
The new case is when $v_j$ has valence $5$, that is, when it had valence $4$ in $X_i$ (before the other fans were completed).
In this case, there are three polygons $P,T, $ and $Q$ in $X_i$ which share a vertex $v_j$, where $T$ is a triangular tile, and there is another polygon $R$ in the fan $F_{j-1}$ that has $v_j$ as a vertex.  (See Figure 8.)

\begin{figure}[h]
 % % Requires \usepackage{graphicx}
  \centering
  \includegraphics[scale=.4]{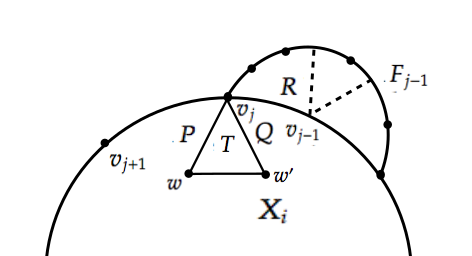}\vspace{.1in}
  \caption{Completing a fan $F_j$ at a vertex $v_j$ of valence $5$ . }
\end{figure}

Thus, around $v_j$, there are polygons $P,T,Q,R$, in that counter-clockwise order. If the corresponding numbers of sides are $x, 3, y$ and $z$, in order to be able to complete a fan at $v_j$ with vertex-type $\mathsf{k}$, we need to show:\vspace{.1in}

\noindent\textit{Claim 5. The quadruple $x3yz$ appears in $\mathsf{k}$.}

\vspace{.05in}
This is where we shall use Condition (B).  Let $v_j$ be the vertex of $T$ as above, and let the other two vertices of $T$ be $w$ and $w^\prime$. Note that $w,w^\prime$ both have vertex-type $\mathsf{k}$ by the inductive hypothesis, since they lie in the interior of $X_i$. Then, since the polygons $T$ and $P$ appear (in counter-clockwise order) around $w$, the pair $3x$ appears in $\mathsf{k}$. Similarly, considering the polygons around $w^\prime$, we see that the pair $y3$ appears in $\mathsf{k}$. Using Condition (A), we deduce, exactly as in a previous claim, that the triple $x3y$ appears in $\mathsf{k}$. Now the vertex $v_{j-1}$ has the polygons $R$ and $Q$ (in counter-clockwise order) around it, so the pair $zy$ also appears in $\mathsf{k}$. Applying the same argument involving Condition (A), we conclude that $3yz$ appears in $\mathsf{k}$. Finally, since the triples $x3y$ and $3yz$ are in $\mathsf{k}$, an application of Condition (B) proves the claim. \vspace{.1in}

This allows a wedge to be added at $v_j$, and divided into polygons, so that the polygons $P,T,Q$ and $R$ are part of a fan $F_j$ of vertex-type $\mathsf{k}$ that is thus completed around $v_j$.

\subsection*{The final fan}

To complete the fan $F_n$ at the remaining boundary-vertex $v_n$, we would need to add a wedge that goes between the fans $F_0$ and $F_{n-1}$, and subdivide into polygons. 

The case when $v_n$ had valence $2$ in $X_i$ is exactly as in the case of completing the final fan in the proof of Theorem \ref{thm1}. 

The remaining case is when $v_0v_n$ is an edge of a triangular tile $T$ and $v_n$ has valence $3$: in this case the polygons in $X_i \cup F_0 \cup F_1 \cup \cdots \cup F_{n-1}$  that  share the vertex  $v_n$ are $P$ (which is part of $F_0$), $T$ and $Q$ (which are part of $X_i$),  and $R$ (which is part of $F_{n-1}$), where $P,T,Q$ and $R$ are in counter-clockwise order around $v_n$.   (See Figure 9.) Suppose the numbers of vertices of $P,Q$ and $R$ are $x,y$ and $z$ respectively. 
Then, by exactly the same argument as in Claim 5, we have that $x3yz$ belongs to the vertex-type $\mathsf{k}$, and hence there is indeed a completion of these four polygons to a fan $F_n$ at $v_n$.  \vspace{.1in}

\begin{figure}
 % % Requires \usepackage{graphicx}
  \centering
  \includegraphics[scale=.4]{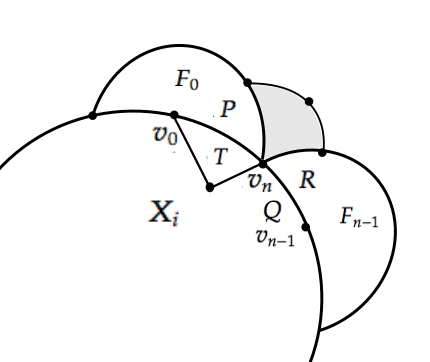}\vspace{.1in}
  \caption{Completing the final fan in the case when $v_0v_n$ is the edge of a triangular tile $T$. }
\end{figure}

This completes the new tiled surface $X_{i+1}$.

\subsection*{Verifying Property $1^\prime$ and Property $2$}

By construction, all interior vertices of $X_{i+1}$ have vertex-type $\mathsf{k}$, and it is easy to see by applying Lemma \ref{toplem} that $X_{i+1}$ is homeomorphic to a closed disk. Moreover, by the remark following Lemma \ref{leml}, we can realize each polygon in $X_{i+1}$ as a regular hyperbolic polygon of side length $l_0$, such that at each vertex the total angle is $2\pi$, and we have a smooth hyperbolic metric in the interior of $X_{i+1}$.  Thus, Property 2 holds for $X_{i+1}$. \vspace{.1in}

Therefore, it only remains to verify Property $1^\prime$, which is where the degree condition  $d\geq 6$ is used.\vspace{.1in}

The key observation is that when the fans $F_j$ (for $0\leq j\leq n$) are added to $X_i$, the following holds:\vspace{.1in}

\noindent\textit{Claim 6. A portion of the wedge added while completing $F_j$ lies on the boundary of $X_{i+1}$. }
\vspace{.05in}

 For example, when a fan $F_j$ is added to a boundary vertex $v_j \in \partial X_i$ having valence $4$ in $X_i$ (see Figure 8), there are already four polygons around $v_j$ in $X_i \cup F_0 \cup F_1 \cup \cdots \cup F_{j-1}$, and the added wedge (to complete the fan $F_j$) needs to have at least one spoke, since the total number of polygons needs to be at least $6$. If $q$ is the endpoint of the first spoke (in counter-clockwise order around $v_j$), then  $F_{j+1} \cap F_j$  cannot contain the portion of the wedge boundary that lies between $q$ and $F_{j-1}$. Hence this portion of the boundary of $F_j$ is on the boundary of $X_{i+1}$.

The same holds for the other cases (when $v_j$ has valencies $2$ or $3$); note that then the added wedge needs to be divided with even more spokes, to have a final valence at least $6$.  This proves Claim 6.\vspace{.1in}

Recall now that Property $1^\prime$ had three parts.\vspace{.1in}

\textit{Proof of (i) and (ii).} A valence $4$ vertex is created in the boundary of $X_{i+1}$ when the fan $F_j$ (for $0\leq j\leq n-1$) has a triangular tile  $T$ in the subdivided wedge, one of whose edges is $v_jv_{j+1}$. In that case, if $q$ is the other vertex of $T$, then $q$ lies in the boundary of $X_{i+1}$, and is disjoint from $F_{j-1}$, by Claim 6 above. Moreover, it has valence $4$ in $X_{i+1}$, since the edge $qv_{j+1}$ will be shared by a polygon in the wedge added at $v_{j+1}$ to complete the next fan $F_{j+1}$.  (See Figure 10.) 
The only other case when a valence $4$ vertex appears in the boundary of $X_{i+1}$ is when the fan $F_0$ has a triangular tile in the added wedge that has side $v_0v_n$. Then, the other vertex $q$ of $T$ lies in the boundary of $X_{i+1}$, and has valence $4$ in $X_{i+1}$ when the wedge (for $F_n$) is added at $v_n$.
In both these cases, the triangular tile $T$ lies in the interior of $X_{i+1}$, and (ii) is satisfied.

In all other cases, when completing a fan, the extreme points of the boundary of any added wedge has valence $3$. Recall that we have proved that Property 2 holds for $X_{i+1}$, and hence its boundary is topologically a circle. If a portion $\gamma$  of the boundary of an added wedge  lies in the boundary of $X_{i+1}$, then $\gamma$ is also a portion of the boundary of a fan. Hence, it has no other edges from it to other parts of $X_{i+1}$, and hence all vertices that lie in $\gamma$ have valence either $2$ or $3$. Moreover, each endpoint of the arc $\gamma$ has valence at least $3$.
This proves (i). \vspace{.1in}

\begin{figure}
 % % Requires \usepackage{graphicx}
  \centering
  \includegraphics[scale=.35]{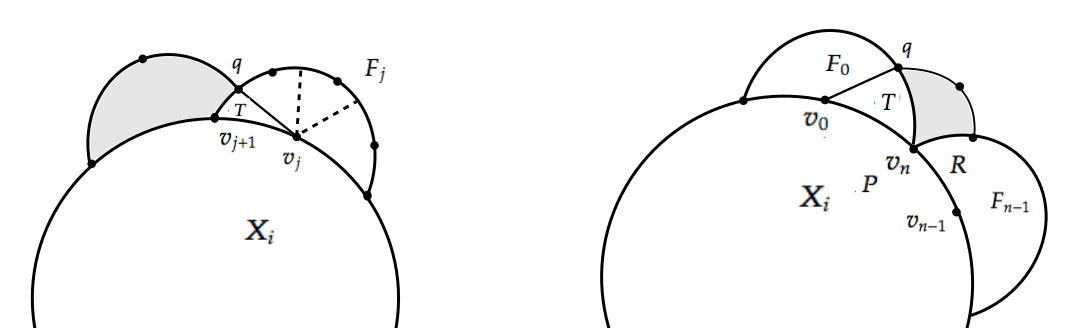}\vspace{.1in}
  \caption{Ways that a valence $4$ vertex $q$ can appear in the boundary of $X_{i+1}$ while completing the fans. Note that a triangular tile $T$ is always involved.}
\end{figure}

\textit{Proof of (iii).} Finally,  recall that the subdivision of the wedge into polygons involves adding spokes, and then, in the case of non-triangular tiles, adding valence $2$ vertices to the  resulting boundary arcs to achieve the desired sizes.
Claim 6 above implies that there is a portion $\gamma$ of an added wedge that lies in the boundary of $X_{i+1}$.  This boundary arc $\gamma$ is either the boundary of  (one or more) triangular tiles tiling the wedge, or there is some polygon in the added wedge having a size greater than $3$.  In the latter case, the subdivision procedure implies that there is a vertex in $\gamma$, and consequently in the boundary of $X_{i+1}$,  that has valence $2$.  
In the former case, there is a boundary edge of $X_{i+1}$ that belongs to a triangular tile.
This proves (iii). \vspace{.1in}

Thus $X_{i+1}$ satisfies Property $1^\prime$, and this completes the inductive step.

\subsection*{Completing the proof} Finally, we verify that we can choose an ordering of the new  boundary vertices $v_0^\prime, v_1^\prime,\ldots ,v_n^\prime$ of $X_{i+1}$ such that $v_0^\prime$ and $ v_n^\prime$ satisfy (a) and (b) stated after Property $1^\prime$. 
In fact, these successive vertices $v_n^\prime, v_0^\prime$ can be chosen to lie along the boundary of the final wedge added while completing $F_n$. Indeed, an extreme point $q$  (in counter-clockwise order) in the boundary of such a wedge, that also belongs to $F_0$, has valence $3$ or $4$ (see, for example, Figure 10). This satisfies (a). The vertex $v$ in the wedge boundary that precedes $q$ either has valence $2$, in case the edge $qv$ belongs to a polygon in the wedge having more than three sides, or else $qv$ is an edge of a triangular tile in the added wedge. In this case, since $d\geq 6$, there is at least one more tile in the added wedge, which the edge $qv$ is adjacent to; hence $v$ has valence exactly $3$. Thus the vertex $v$ satisfies (b), and the vertices $q$ and $v$ can be taken to the first and last vertices ($v_0^\prime$ and $v_n^\prime$) respectively, in our new counter-clockwise ordering of the boundary vertices of $X_{i+1}$. 

Thus, we get a sequence of nested tiled surfaces (each a closed hyperbolic disk) 
\begin{equation*}
X_0 \subset X_1\subset \cdots \subset X_i \subset X_{i+1} \subset \cdots
\end{equation*}
such that any interior vertex of  $X_i$  has vertex-type $\mathsf{k}$. Lemma \ref{exh} still applies (its proof is independent of the hypotheses of Theorems \ref{thm1} and \ref{thm2}),  and this sequence  of regions exhausts the hyperbolic plane, defining the desired semi-regular tiling.  This proves Theorem \ref{thm2}.

\section{Uniqueness and uniformity: Proofs of Theorem \ref{thm-uniq} and Corollary \ref{cor1}} 

%If $\mathsf{k}$ satisfies the property that any pair of consecutive elements of $\mathsf{k}$ determines the rest of the cyclic tuple uniquely, then as before, the completions of the fans $F_j$ for $0\leq j\leq n$ is uniquely determined. This is because at least two polygons are in place around the boundary vertex, and therefore allows a unique way of adding polygons to complete the fan. The construction of the semi-regular tiling is thus determined uniquely. 

\subsection*{Proof of Theorem \ref{thm-uniq}}

Let $T$ be the semi-regular tiling with vertex-type $\mathsf{k}$ obtained from Theorem \ref{thm1} or \ref{thm2}. Assume that the vertex-type $\mathsf{k}$ satisfies the additional property in the statement of Theorem \ref{thm-uniq}. 

 The uniqueness statement then follows from the observation that (a) the semi-regular tiling $T$ can be thought of as arising from our construction, and (b) if the vertex-type $\mathsf{k}= [k_1,k_2,\ldots, k_d]$ satiisfies the hypothesis of the theorem, then there is a unique way of completing the fan $F_j$ for each $0\leq j\leq n$. This is because at any such step, the partial fan that was already at $v_j$, had (at least) two polygons $P$ and $Q$ already in place around it.  If $x$ and $y$ are the sizes of $P$ and $Q$ respectively, then this determines a pair $xy$ that appears in $\mathsf{k}$. Then by the assumed property of $\mathsf{k}$, there is a unique sequence of polygons that can follow $P$ and $Q$ in the final fan. This determines a unique way of choosing the subdivision of the added wedge that determines these polygons. 
 Hence the $i$-th stage of the construction (expanding from $X_{i-1}$ to $X_{i}$)  is determined uniquely for each $i\geq 1$, and thus the final tiling $X_\infty$ is determined uniquely. 
 
 The uniformity statement in fact follows immediately from the fact that the construction of the sub-tilings $X_0\subset X_1\subset X_2\subset \cdots$ is determined uniquely, by our assumption on $\mathsf{k}$.  In what follows, we give a more geometric argument. 
 
%note that  in the observation (a) above, the initial vertex $V$ in our construction can be chosen to be any vertex of $T$.  So if $v_1$ and $v_2$ are two distinct vertices of $T$, then the construction determines two identical tilings, each isomorphic to $T$ via isomorphisms $\phi_1$ and $\phi_2$  that take the vertex $V$ to $v_1$ and $v_2$ respectively.  Then the composition $\phi_2\circ \phi_1^{-1}$ is an automorphism of $T$ that takes $v_1$ to $v_2$. This proves that the automorphism group of $T$ is vertex-transitive. In what follows, we shall give a more explicit, and geometric, argument for uniformity:
 
Let $V$ be an initial vertex, and $X_i$, $i\geq 0$ be the sequence of surfaces, and $X_{\infty}$ their union, as in the construction in \S2 and \S3.  Since $X_\infty$ is isometric to $\mathbb{H}^2$ (by Lemma \ref{exh}), we can consider each $X_i$ (for $i\geq 0$) as an isometrically embedded subset of $\mathbb{H}^2$. Let $U:=v_0, \dots, v_n$ be the vertices in $\partial X_0$ (ordered in the counter-clockwise direction), where $VU$ is an edge (equivalently, $\deg_{X_0}(U) =3$). Let  $P= V\mbox{-} U\mbox{-}v_1\mbox{-}\cdots \mbox{-}v_{p-2}\mbox{-}V$, $Q= U\mbox{-} V\mbox{-}v_{n-q+3}\mbox{-}\cdots \mbox{-}v_{n}\mbox{-}U$  be the polygons (tiles) containing the edge $VU$, of sizes $p$ and $q$ respectively. Thus, the vertex-type {$\mathsf{k}$}  can be expressed as $ [p, q, q_3, \dots, q_d]$, where note that $q_3, \ldots, q_d$, in that order, are uniquely determined by the pair $pq$.  The sizes of the polygons containing $V$ in the clockwise direction around $V$ are then $p, q, q_3, \dots, q_d$ respectively, and these are also the sizes of the polygons containing $U$ in the counter-clockwise direction. The last set of polygons forms the complete fan around $U$, which is also the 0-th layer $Y_0$ when we perform the construction in \S2 or \S3 starting with the initial vertex $U$.  Like the unique sequence of layers $X_0\subset X_1\subset X_2\subset \cdots $ with initial vertex $V$, there exists a unique sequence of layers $Y_0\subset Y_1\subset Y_2\subset \cdots$ with initial vertex $U$, obtained from our construction.  Once again, we consider each $Y_i$ (for $i\geq 0$) as a subspace of $\mathbb{H}^2$. 

Then the hyperbolic reflection $\rho:\mathbb{H}^2 \to \mathbb{H}^2$ which fixes the edge $UV$ and interchanges $U$ and $V$ is an isomorphism (taking vertices to vertices and edges to edges) between the tiled regions $X_0$ and $Y_0$. Moreover, $\rho(P) = P$, $\rho(Q) = Q$, $\rho(v_1) = v_{p-2}$, $\rho(v_n) = v_{n-q+3}$. Clearly,  the image of the fan at $v_1$ (which is a subset of $X_1$) under $\rho$ would be the fan at $v_{p-2}$ (which is a subset of $Y_1$). Continuing this way, considering the fan at each of the vertices $v_2, v_3,\ldots, v_n$, we see that the image of $X_1$ under $\rho$ is $Y_1$. By the same way, we see that $\rho(X_2)=Y_2$, $\rho(X_3)=Y_3, \ldots$, that is, $\rho(X_i) = Y_i$ for each $i\geq 0$. Thus, $\rho$ is an isomorphism between $X_{\infty}=\cup_{i\geq 0}X_i$ and $Y_{\infty}=\cup_{i\geq 0}Y_i$.   Now, observe that $Y_0 \subset X_1$ and $X_0 \subset Y_1$. Moreover, by the uniqueness of $X_i$'s and $Y_i$'s, it follows that $Y_1 \subset X_2$, $Y_2 \subset X_3, \dots$ and $X_1 \subset Y_2$, $X_2 \subset Y_3, \dots$, that is, $X_{i} \subset Y_{i+1}$ and $Y_i \subset X_{i+1}$ for each $i\geq 0$. 
This implies  $X_{\infty} = Y_{\infty}$ and $\rho$ is an automorphism of the tiling $T=X_{\infty} = Y_{\infty}$  such that  $\rho(V) = U$, $\rho(U)=V$.  

Since $U$ is any neighbour of $V$, this argument shows that any two adjacent vertices of $T$ are in the same $\mbox{Aut}(T)$-orbit. Since the $1$-skeleton of $T$ is connected, it follows that all the vertices form one $\mbox{Aut}(T)$-orbit. Therefore, the action of $\mbox{Aut}(T)$ on the vertex set of $T$ is transitive, that is, the semi-regular tiling $T$ is uniform.  This completes the proof.\hfill$\qed$
 
%In the case that the vertex-type $\mathsf{k}$ satisfies the pair of conditions (A1) and (A2)  of property (ii) in the statement of the theorem, the argument is similar.  It is easy to verify that these conditions imply condition (A). Once again, at the $i$-th step of the construction, there is a unique way of completing the partial fan at each boundary vertex in $\partial X_i$.  In fact, in the construction, we can ensure that  the polygons corresponding to the integers in $\mathsf{k}$ appear in the same cyclic order (either clockwise or counter-clockwise) around every vertex. If we assume this 

\subsection*{Proof of Corollary \ref{cor1}} 

By Lemma \ref{cor-leml},  it is enough to show that two semi-regular tilings $T$ and $T^\prime$ with the same vertex-type $\mathsf{k} =[p^q]$, where $\frac{1}{p} + \frac{1}{q} < \frac{1}{2}$, are equivalent.  Note that the above inequality arises from the angle-sum condition \eqref{asum}. 

It is easy to see that the uniqueness criterion in Theorem \ref{thm-uniq} is satisfied by the cyclic tuple $\mathsf{k}=[p^q]$.

If $p =3$, then the inequality arising from the angle-sum condition implies that $q\geq 7$. Hence, in this case, the hypotheses of Theorem \ref{thm2}, and the uniqueness criterion in Theorem \ref{thm-uniq}, are satisfied by $\mathsf{k}=[p^q]$, and we deduce that the two tilings are equivalent.

If both $p,q\geq 4$ the hypotheses of Theorem \ref{thm1} and  the uniqueness criterion in Theorem \ref{thm-uniq} are satisfied by $\mathsf{k}$, and we similarly deduce that the two tilings are equivalent.

Finally, if $q=3$, then $p\geq 7$, and the tilings are the {duals} to semi-regular tilings, each with vertex-type $[3^p]$. (Here, the \textit{dual} of a semi-regular tiling is constructed by taking vertices at the incenters of the original tiles, and connecting any pair of vertices in adjacent tiles by a geodesic edge.) The latter tilings are equivalent, as noted above, and hence so are their duals. \hfill$\qed$

% \textit{Remark.}  By our construction, the semi-regular tilings obtained in Theorems \ref{thm1} and \ref{thm2} has the additional property that  the polygons corresponding to the integers in $\mathsf{k}$ appear in the same order (either clockwise or anti-clockwise) around every vertex.\vspace{.1in}

 \section{Examples of non-uniqueness}
 
 In this section we give examples of distinct tilings with the same vertex-type.

 \begin{figure}[h]
 % % Requires \usepackage{graphicx}
  \centering
  \includegraphics[scale=.4]{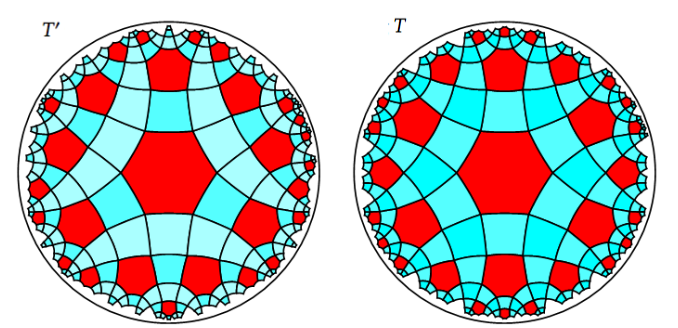}\vspace{.1in}
  \caption{Spot the difference: these are distinct semi-regular tilings with identical vertex-type [4,4,4,6]. }
\end{figure}

\subsection*{Infinitely many distinct tilings}

Consider the cyclic tuple  $\mathsf{k}= [4,4,4,6]$.  Note that such a cyclic tuple does \textit{not} satisfy the criterion for uniqueness: the pair $44$ can be continued both as $[\textbf{4},\textbf{4},4,6]$ as well as $[\textbf{4},\textbf{4},6,4]$. 

In this case, there is a semi-regular tiling $T$ of vertex-type $\mathsf{k}$ such that $T$ is invariant under an action of a discrete subgroup $\Gamma < \text{PSL}_2(\mathbb{R})$, that acts transitively on the hexagonal tiles.  (See the tiling $T$ on the right in Figure 11.) Here, $\Gamma$ is generated by the three  hyperbolic translations, together with their inverses, that take the central red hexagon to the six nearest hexagons lying on the three axes passing through (and orthogonal to) the three pairs of opposite sides of the central hexagon.  

Notice that it has a $\Gamma$-invariant collection $\mathcal{R}$ of bi-infinite rows of squares $\{R_\gamma\vert \gamma \in \Gamma\}$  (see the rows of blue squares in Figure 11). Any such row $R_\gamma$ has adjacent layers $L_+$ and $L_-$ that comprise alternating hexagons and squares.

Then, a tiling $T^\prime$ that is \textit{distinct} from $T$ is obtained by shifting one side of each $R_\gamma$ relative to the other.  For example, performing this shift for three such bi-infinite rows adjacent to alternating sides  (left, right, and bottom) of the central red hexagon in $T$ produces a new tiling $T^\prime$ (on the left in Figure 11). This is clearly not equivalent to $T$ since, for example, $T^\prime$ has a local configuration comprising a pair of hexagons with a chain of three squares between them, which is absent in $T$. 

The same technique works for the vertex-type $[4,4,4,n]$ for $n>4$. There is a semi-regular tiling with this vertex-type which has an infinite collection $\mathcal{R}$ of bi-infinite rows of squares. 

Now consider a subset $\mathcal{S}$ of the collection of rows $\mathcal{R}$ that are ``sufficiently far apart", that is, each pair of rows in the subset are disjoint, and there is no hexagon having two sides belonging to the two rows in the pair. 
The relative shift as above can then be performed simultaneously for the rows in $\mathcal{S}$;  for each row, the change in the tiling is shown below (Figure 12.) The assumption of the rows being ``sufficiently far apart" ensures that these shifts are independent of each other. Since there are infinitely many such subsets of $\mathcal{R}$ that are not equivalent under the symmetries of $T$, we obtain infinitely many distinct tilings. Once again, the fact that they are distinct can be shown by producing local configurations that are unique to each tiling.

\begin{figure}[h]
 % % Requires \usepackage{graphicx}
  \centering
  \includegraphics[scale=.27]{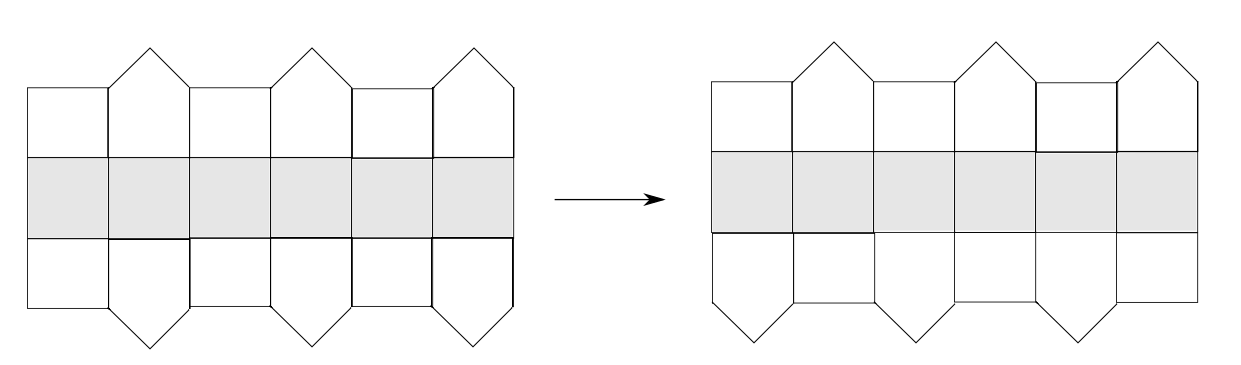}\vspace{.1in}
  \caption{A relative shift of the tiles on either of the row $R_\gamma$ (shown shaded) produces a different tiling with the same vertex-type $[4,4,4,5]$.}
\end{figure}

\subsection*{Other examples} 

Note that our construction in \S2 and \S3  can be done starting with any initial tiled surface $X_0$  that satisfies Property 1 (in the case  that no tile is triangular and $d\geq4$) or Property $1^\prime$ (in the case $d\geq 6$) and Property 2, together with the property that each interior vertex has the same vertex-type $\mathsf{k} = [k_1,k_2,\ldots ,k_d]$.  

%This gives a couple of different ways of distinct semi-regular tilings with the same vertex-type arise:\vspace{.1in}(a)

 If $\mathsf{k}$ satisfies the hypotheses of Theorem \ref{thm2}, but has a pair $xy$ of consecutive elements which can be completed to the same cyclic tuple in two \textit{different} ways, and these choices show up while completing the fans, then the final semi-regular tilings could be different. For an example, when the vertex-type is $\mathsf{k} = [4, 3,3,3,4,3]$, the pair $33$ can be continued both as $[\textbf{3},\textbf{3},3, 4,3,4]$ as well as $[\textbf{3},\textbf{3},4,3,4,3]$, and this choice, used judiciously in the tiling construction, gives rise to two distinct semi-regular tilings -- see \url{https://en.wikipedia.org/wiki/Snub_order-6_square_tiling}.   

%\begin{figure}[h]
% % % Requires \usepackage{graphicx}
 % \centering
 % \includegraphics[scale=.42]{choice4.png}\vspace{.1in}
 % \caption{The unshaded tiles form a fan around $V$ with vertex-type $\mathsf{k} = [4,3,3,3,4,3]$. The figures show two ways of completing another fan with the same vertex-type at a boundary vertex $u$. This choice, used judiciously in the tiling construction, could give rise to two distinct semi-regular tilings.}
%\end{figure}

%(b) Suppose $X_0$ and $X_0^\prime$ are two different initial tiled regions such that the tiling $T$ obtained by extending $X_0$  does not have an subset isometric to $X_0^\prime$. Then the tiling $T^\prime$ obtained by extending $X_0^\prime$ is necessarily distinct from $T$. %This gives a more general way to construct a host of distinct pairs of semi-regular tilings with the same vertex-type (see Figure 14).

\section{Degree-3 tilings: Proof of Theorem \ref{thm0}}

In this section we prove Theorem \ref{thm0}.  Our proof relies on the following result of \cite{Edmondsetal} that was already mentioned in the Introduction:

\begin{thm}[Theorem 1.3 of \cite{Edmondsetal}]   \label{Kul2} 
Let $M$ be a closed orientable surface, and let $m_1, \dots, m_d$, $R$, $E$, $V_1, \dots, V_d$ be positive integers satisfying 
\begin{align*}
2E=dR; ~ 2m_iV_i=R, ~ i=1 ..... d; \mbox{ and } R-E+ \sum_{i=1}^{d}V_i =\chi(M). 
\end{align*}
Then there is a tiling of $M$ into $R$ \, $d$-gons, with $E$ edges, and $\sum_i V_i$ vertices,  $V_i$ of valence $2m_i$ $(i=1, \dots, d)$, such that each region has vertices of valence $2m_1, \dots, 2m_d$, up to cyclic order.
\end{thm}

In what follows,  the \textit{dual} of a tiling (or map) $X$ on a surface $M$, is defined as follows:
Take a new vertex $v_F$ in the interior of each face $F$ of $X$. If $e$ is an edge of $X$ then $e$ is the intersection of exactly two faces $F$ and $G$. Join $v_F$ with $v_G$ by an edge $\tilde{e}$ inside $F\cup G$. We choose the edges such that intersection of two edges $\tilde{d}$ and $\tilde{e}$ is either empty or a common end point. Then the new vertices and new edges define a map on $M$ called the {\em dual} of $X$ and is denoted by $X^{\ast}$. 

This combinatorial definition coincides with the notion of the \textit{dual} of a semi-regular tiling introduced in the proof of Corollary \ref{cor1} at the end of \S4. 

% is constructed by taking vertices at the incenters of the original tiles, and connecting any pair of vertices in adjacent tiles by a geodesic edge.

\begin{cor} \label{cor2} \label{cor_Kul2} \label{eventype} 
For $d\geq 3$, consider integers $m_1, \dots, m_d \geq 2$, such that the vertex-type $\mathsf{k} = [2m_1, \dots, 2m_d]$ satisfies the angle-sum condition \eqref{asum}.  Then there exists a semi-regular tiling on the hyperbolic plane of vertex-type $\mathsf{k}$. 
\end{cor}
\begin{proof} It can be checked that the angle-sum condition reduces to $\frac{1}{m_{1}} +  \cdots + \frac{1}{m_{d}} < d-2$. 
Consider the positive  integers $R = 4m_1\cdots m_d$,  $E=dR/2$, $V_i=R/(2m_i)$, $1\leq i\leq d$. Let   $\chi = R - E + \sum_{i=1}^d V_i$. Then $\chi = R(1 - \frac{d}{2} + \frac{1}{2}\sum_{i=1}\frac{1}{m_i}) = -\frac{R}{2}((d-2) - \sum_{i=1}^d\frac{1}{m_i})$. So, $\chi$ is an a negative even integer. Let $M$ be the orientable closed surface with Euler characteristic $\chi(M) = \chi$. 
Then the surface $M$, and integers $m_1, \dots, m_d$, $R$, $E$, $V_1, \dots, V_d$ satisfy the hypothesis of Theorem \ref{Kul2}. Therefore, by Theorem \ref{Kul2}, there exists a tiling $X$ of $M$ into $R$ regular $d$-gonal faces, with $E$ edges, and $\sum_{i=1}^d V_i$ vertices, $V_i$ of valence $2m_i$ $(i=1, \dots, d$,  with the understanding that the $m_i$s need not be distinct), such that each face has vertices of valences $2m_1, \dots, 2m_d$, up to cyclic order. 

Let $X^{\ast}$ be the dual of $X$, as defined above. Then by Lemma \ref{cor-leml}, $X^{\ast}$ has a geometric realization as a semi-regular tiling of degree $d$ and of vertex-type $[2m_1, \dots, 2m_d]$ on the closed surface $M$. Since $\chi(M)$ is negative, the hyperbolic plane is the unversal covering of $M$. The semi-regular tiling $X^{\ast}$ on $M$ then lifts to a semi-regular tiling $Y :=\widetilde{X^{\ast}}$  of the same vertex-type on the hyperbolic plane. 
\end{proof}

\begin{proof}[Proof of Theorem \ref{thm0}] 

Our proof divides into several cases which we handle separately.\vspace{.1in}

\textit{Case 1: $\mathsf{k} = [p,p,p]$.} Note that for the angle sum \eqref{asum} to be satisfied, we have $p\geq 7$.  Semi-regular tilings with this vertex-type $\mathsf{k}$ exist, as they are dual to the semi-regular tilings $[3^p]$ which exist by the Fuchsian triangle-group construction mentioned in the introduction.  \vspace{.1in}

\begin{figure}
 % % Requires \usepackage{graphicx}
  \centering
  \includegraphics[scale=.35]{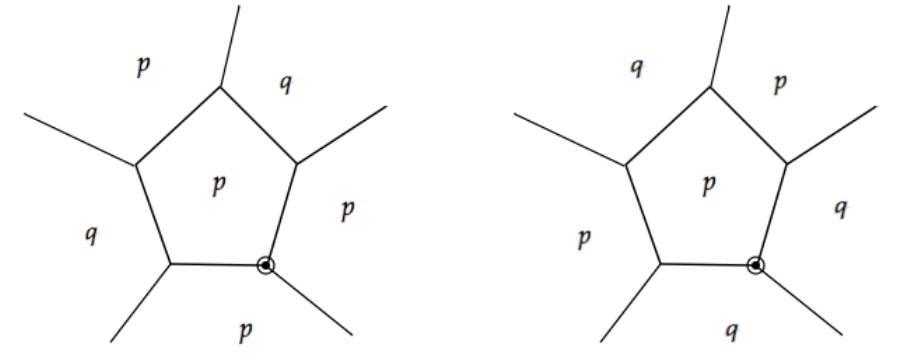}\vspace{.1in}
  \caption{When $p$ is odd, there cannot be a semi-regular tiling with vertex-type $[p,p,q]$: the offending vertex is circled in the two possibilities.}
\end{figure}

\textit{Case 2: $\mathsf{k} = [p,p,q]$ where $p\neq q$, and $p$ is odd.}  In this case, suppose there is a  semi-regular tiling with vertex-type $\mathsf{k}$.  Consider a $p$-gon in this tiling, with vertices $v_0, v_1,\ldots v_{p-1}$, and edges $e_i$ between $v_i$ and $v_{i+1}$ for $0\leq i\leq p-1$ (considered modulo $p$).  Since each of these vertices has degree $3$, the edges  alternately share an edge with a $p$-gon and $q$-gon, respectively. However, if $p$ is odd, then there is a vertex with three $p$-gons or two $q$-gons and a $p$-gon around it, which contradicts the fact that the vertex-type is $\mathsf{k}$. (See Figure 13.) Thus, there can be no semi-regular tiling with vertex-type $\mathsf{k}$.  This argument also appears in \cite{Datta1b}, Lemma 2.2 (i), in the context of maps on surfaces.\vspace{.1in}

\textit{Case 3: $\mathsf{k} = [p,p,q]$ where $p\neq q$, and $p$ is even.}  Let $p=2n$. Then the angle-sum in \eqref{asum} can be easily seen to yield that $\frac{1}{n} + \frac{1}{q} < \frac{1}{2}$. Note that this is exactly the same condition that implies the existence of an $[n^q]$ tiling. We can in fact construct a semi-regular tiling with vertex-type $\mathsf{k}$ by modifying an $[n^q]$ tiling $T_0$ in the following way: replace each vertex in $T_0$ by a $q$-gon, which has vertices along the edges of $T$. The tiles of $T_0$ are now $2n$-gons, since each such tile acquires an extra edge from the $q$-gon added at each vertex, and there are $n$ vertices. (This is well-known procedure  called ``truncation" -- see Chapter VIII of \cite{Coxeter}.) Although this construction is a \textit{topological tiling}, we can replace them by regular hyperbolic polygons since the angle-sum condition is satisfied (\textit{c.f.} Lemma \ref{cor-leml}).  (See Figure 14.) \vspace{.1in}

\begin{figure}
 % % Requires \usepackage{graphicx}
  \centering
  \includegraphics[scale=.3]{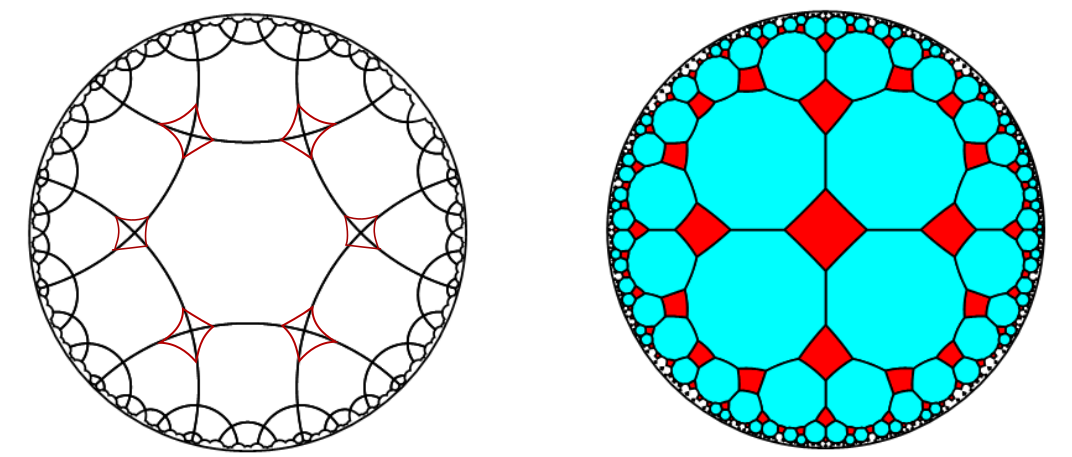}\vspace{.1in}
  \caption{ A  semi-regular tiling with vertex-type $\mathsf{k} = [12,12,4]$ (figure on the right) is obtained by ``truncating" each tile of a semi-regular tiling with vertex-type $[6^4]$  (figure on the left). }
  \end{figure}

\textit{Case 4: $\mathsf{k} = [p,q,r]$ where $p,q,r$ are distinct.}  By Lemma 2.2 (ii) of \cite{Datta1b}, all three $p$, $q$ and $r$ have to be even, say $p=2\ell$, $q=2m$, $r=2n$, where $\ell, m, n\geq 2$ are distinct.
In this case, the angle sum condition is satisfied if $\frac{1}{\ell} + \frac{1}{m} + \frac{1}{n} < 1$.
By Corollary \ref{cor_Kul2}, such a semi-regular tiling exists on the hyperbolic plane. \vspace{.1in}

This covers all possibilities for a triple $\mathsf{k}$, and completes the proof. 
\end{proof}

%\section{Further questions}

\medskip

\noindent \textbf{Acknowledgements.}  The first author is supported by SERB, DST (Grant No. MTR/2017/ 000410). The second author acknowledges the SERB, DST (Grant no. MT/2017/000706) and the Infosys Foundation for their support.  The authors are also supported by the UGC Centre for Advanced Studies. The authors thank the referee for several suggestions that significantly improved this article. Several figures in this article were made using the {\tt L2Primitives} and {\tt Tess} packages for Mathematica, available online at the Wolfram Library Archive.

\bibliographystyle{alpha}
\bibliography{Tiling-refs-2}

\end{document}